\RequirePackage{fix-cm}
\documentclass[smallextended]{svjour3}       
\smartqed  
\usepackage{graphicx}
\usepackage{booktabs}
\usepackage[LGR,T1]{fontenc}
\usepackage[greek,english]{babel}
\usepackage{amsmath}
\usepackage{amssymb,amsfonts,textcomp}
\usepackage{array}
\usepackage{hhline}
\usepackage{graphicx}
\usepackage{booktabs}
\usepackage[round, sort]{natbib}
\usepackage{pstricks}

\usepackage{floatrow}

\newcommand{\vp}{v^+}
\newcommand{\vm}{v^-}
\newcommand{\fpen}{f^{pen}}
\newcommand{\xb}{\bar{x}}
\newcommand{\yb}{\bar{y}}
\newcommand{\xn}{x^0}
\newcommand{\yn}{y^0}
\newcommand{\modf}{\tilde{f}}

\begin{document}

\title{A bicriteria perspective on $L$-Penalty Approaches -- A corrigendum to Siddiqui and Gabriel's $L$-Penalty Approach  
for Solving MPECs
}

\titlerunning{Corrigendum of Siddiqui and Gabriel's $L$-penalty approach}        

\author{D{\"a}chert, Kerstin
\and
Siddiqui, Sauleh
\and
Saez-Gallego, Javier
\and
Gabriel, Steven A.
\and
Morales, Juan Miguel
}

\institute{
K. D{\"a}chert \at
University of Duisburg-Essen, Faculty of Economics and Business Administration, 
		    Chair for Management Science and Energy Economics\\		    
		    Berliner Platz 6-8, 45127 Essen, Germany\\
              \email{kerstin.daechert@uni-due.de}
\and
S. Siddiqui \at
              Departments of Civil Engineering and Applied Mathematics \& Statisitics, Johns Hopkins University \\
3400 N Charles Street Baltimore, MD 21218  \\
\email{siddiqui@jhu.edu}
\and
J. Saez-Gallego \at
              Siemens Wind Power A/S \\
Borupvang 9, 2750 Ballerup, Denmark \\
\email{javiersaezgallego@gmail.com}
\and
S.A. Gabriel \at
               Department of Mechanical Engineering, 
               Applied Mathematics, Statistics, and Scientific Computation Program,
               University of Maryland, College Park, MD, 20742-3021 USA\\
              \email{sgabriel@umd.edu}
\and
J.M. Morales \at
University of M{\'a}laga\\
Department of Applied Mathematics, School of Industrial Engineers, room 2.556L-B, M{\'a}laga, Spain\\
              \email{juan.morales@uma.es}
}

\date{April 2nd, 2017}

\maketitle

\begin{abstract}
This paper presents a corrigendum to Theorems~2 and~3 in \citet{Gabriel2013}. 
In brief, we revise the claim that their $L$-penalty approach yields a solution satisfying complementarity for any positive value of~$L$, in general.
This becomes evident when interpreting the $L$-penalty method as a weighted-sum scalarization of a bicriteria optimization problem. 
We also elaborate further assumptions under which the $L$-penalty approach yields a solution satisfying complementarity. 

\end{abstract}

\section{Preface}
In this paper we follow the notation of \citet{Gabriel2013}. 
However, in order to make this paper self-contained, we briefly repeat the relevant formulas from \citet{Gabriel2013} in the next section indicating by an asterisk those coming from the original paper.  

\section{Introduction}
Following the approach in \citet{Gabriel2013}, we consider a mathematical program with equilibrium constraints given by
\begin{align} \label{eq:compl}
\text{min} \ & f(x,y) \notag
\\
\text{s.t.} \  & (x,y)\in \Omega \tag{$1^{\star}$} 
\\
& y\in S(x) \notag
\end{align}
where the continuous variables $x \in \mathcal{R}^{n_x}$ and $y \in \mathcal{R}^{n_y}$
are, respectively, the vector of upper-level and lower-level variables,
$f(x,y)$ is the upper-level single-objective function,  
$\Omega$ is the joint feasible region between these sets of variables and 
$S(x)$ is the solution set of the lower-level problem that can take the form of an optimization problem, a nonlinear complementarity problem (NCP), or a variational inequality problem \citep{luo1996mathematical}. 

The main focus of \citet{Gabriel2013}  is when $S(x)$ is the solution set of an NCP.
Then \eqref{eq:compl} can be rewritten as
\begin{align} \label{eq:mpecorig}
\text{min} \ & f(x,y) \notag
\\
\text{s.t.} \ & (x,y)\in \Omega \notag
\\
& y \ge 0 \tag{$3^{\star}$}
\\ 
& g(x,y)\ge 0\hfill\null  \notag
\\
& y^{\top} g(x,y)=0 \notag
\end{align}
where 
$g(x,y): \mathcal{R}^{n_x} \times \mathcal{R}^{n_y} \to \mathcal{R}^{n_y}$
is a vector-valued function. 
We often make use of the shorthand notation $f=f(x,y)$, $g=g(x,y)$ for convenience. 

The set $y^{\top} g=0$ is non-convex in $x,y$ and can be computationally challenging to find even if $g$ is linear. 
In this case, the MPEC~\eqref{eq:mpecorig} can be reformulated using  
Schur's decomposition,
see 
($6^{\star}$), ($8^{\star}$) and ($9^{\star}$) of \citet{Gabriel2013}, 
where, in brief, new variables $u$ and $v$ are introduced by setting 
$u=(y+g)/2$  
and
$v=(y-g)/2$.
Then $y=u+v$ and $g=u-v$, thus, 
$y_{i}g_{i}=(u_{i}+v_{i})(u_{i}-v_{i})=u_{i}^2-v_{i}^2$
for every $i=1,\dots,{n_y}$.
Since $y \geq 0$ and $g\geq 0$, it follows that $u \geq 0$.
Hence, $u_{i}^2-v_{i}^2=0$ is equivalent to $u_{i}-|v_{i}|=0$.
If $v$ is replaced by two non-negative variables, i.e., $v=\vp-\vm$ with $\vp,\vm \geq 0$, where (component-wise) at most one is non-zero, the absolute value can be expressed as $|v|=\vp+\vm$.  
The corresponding reformulation with SOS~1 variables (special ordered sets of type~1, 
defined as a set of non-negative variables 
of which at most one can take a strictly positive value) reads
\begin{align} \label{eq:sos}
\min \ & f(x,y)  \notag
\\
 \text{s.t.} \  & (x,y)\in \Omega \notag
\\
& y \ge 0 \notag
\\
& g(x,y)\ge 0 \notag
\\
& u-(v^+ + v^-) = 0 \tag{$9^{\star}$}
\\
& u=\frac{y+g(x,y)}{2}	\notag
\\
& v^+ - v^- = \frac{y-g(x,y)} 2	\notag
\\
& \text{where } v^+ , v^- \text{ are SOS~1 variables.} \notag
\end{align}
Formulation~\eqref{eq:sos} is a viable way to solve the MPEC~\eqref{eq:compl} as shown in \citet{Gabriel2013}.
They also propose an $L$-penalty method of the form
\begin{align} \label{eq:penL}
\min \ & f(x,y) + \sum_{i=1}^{n_y} L_i(v_i^+ + v_i^-) \notag
\\
 \text{s.t.} \  & (x,y)\in \Omega \notag
\\
& y \ge 0 \notag
\\
& g(x,y)\ge 0 \notag
\\
& u-(v^+ + v^-) = 0 \tag{$10^{\star}$}
\\
& u=\frac{y+g(x,y)}{2}	\notag
\\
& v^+ - v^- = \frac{y-g(x,y)} 2	\notag
\\
& \text{where } v^+ , v^- \text{ are non-negative variables,} \notag
\end{align}
and $L_{i}>0$ for all $i =1,\dots,n_{y}$.  
Compared to~\eqref{eq:sos} the SOS~1 property is relaxed and instead a new term weighted by parameters $L_{1},\dots,L_{n_{y}}$ is added to the objective function.
In what follows we use a scalar 
$L=L_{1}=\dots=L_{n_{y}}$
whenever a distinction by different parameter values is not necessary.

Since we can replace $\vp_{i}+\vm_{i}$ in the objective function by 
$u_{i}=(y_{i}+g_{i})/2$,
the auxiliary variables $u, \vp, \vm$ 
are not required in \eqref{eq:penL} and can thus be removed.
This yields the simplified but equivalent $L$-penalty formulation~\eqref{eq:penL2}.
\begin{align} \label{eq:penL2}
\min \ & f(x,y) + \sum_{i=1}^{n_y} L_i \cdot \frac{y_i + g_i(x,y)}{2} \notag
\\
 \text{s.t.} \  & (x,y)\in \Omega \notag
\\
&y \ge 0 \tag{$10^{\star}b$}
\\
&g(x,y) \ge 0. \notag
\end{align}
Theorem~2 of~\citet{Gabriel2013} states that if problem~\eqref{eq:sos} has a solution and if the KKT-conditions are both necessary and sufficient for~\eqref{eq:penL}, then for any $L_i>0$ and for each $i$, problem~\eqref{eq:penL} has a solution where at most one of $(\vp)_i$ and $(\vm)_i$ is nonzero 
(i.e., the SOS~1 property holds).
Translated to~\eqref{eq:penL2} this implies that for any $L_i>0$ and for each $i$, either $y_{i}$ or $g_{i}$ or both are zero, i.e., complementarity is satisfied. 
This is because $\vp_{i}=\frac{y_{i}}{2}$ and $\vm_{i}=\frac{g_{i}}{2}$. 
Unfortunately, Theorem~2 does not hold. 
The same applies for Theorem~3 of~\citet{Gabriel2013}, which states that 
if the feasible set of~\eqref{eq:penL} is non-empty and if the maximum of 
$\sum_{i=1}^{n_y} \vp_i + \vm_i=\sum_{i=1}^{n_y} (y_i + g_i)/2$ 
over this set exists, 
then there exists an $L>0$ so that for all positive $\hat{L} \leq L$ a solution to problem~\eqref{eq:penL} with penalty~$\hat{L}$ also solves~\eqref{eq:sos}, i.e., satisfies complementarity. 

The rest of the corrigendum is structured as follows. 
In Section~\ref{sec:counterexamples} we present counter-examples to Theorems~2 and~3 of~\citet{Gabriel2013}.
In Section~\ref{sec:bicrit} the failure of the theorems is analyzed from a bicriteria perspective.
Section~\ref{sec:repair} contains a new theorem that guarantees solutions of the proposed $L$-penalty method satisfying complementarity under certain assumptions.  
Conclusions are summarized in Section~\ref{sec:conclusions}.

\section{Counter-examples } \label{sec:counterexamples}

\subsection{Counter-example to Theorem~2}
\begin{example}[Non-complementarity for certain values of $L$] \label{ex:1}
Consider the following instance of Problem~\eqref{eq:compl} with a parameter $K>0$:
\begin{subequations} \label{eq:bil}
\begin{alignat}{2}
\min \; & f(x,y)=x_3 + x_4 + x_5 + x_6  \\
s.t. \; & x_3 - x_4  = x_1-7 \label{eq:abs1}\\
         &x_5 - x_6   = x_2-3   \label{eq:abs2}\\
         &y_1            = 3 \label{eq:stat1} \\
        &-4 + y_1 + y_2 - K y_3 = 0 \label{eq:stat2} \\
       &g_{1}(x,y) = 10 - x_1 - x_2  \geq 0  \label{eq:comp1} \\
      &g_{2}(x,y) = x_7 - x_2   \geq 0 \label{eq:comp3}\\
     &g_{3}(x,y) = Kx_2   \geq 0 \label{eq:comp2} \\
    &y^{\top} g = 0 \label{eq:comp}\\
    &x,y \geq 0
\end{alignat}
\end{subequations}
The unique optimal solution to \eqref{eq:bil} is 
$x= \left(7, 3, 0, 0, 0, 0, 3 \right), y = \left(3, 1, 0 \right)$, 
which can be seen as follows.
Since $x$ is non-negative, the objective function value is bounded below by zero which is attained if and only if $x_3 = x_4 = x_5 = x_6 = 0$, assuming that it is part of a feasible solution. 
Then, \eqref{eq:abs1} yields $x_1 = 7$ and 
\eqref{eq:abs2} yields $x_2 = 3$. 
Furthermore, 
\eqref{eq:comp1} yields $g_{1} = 0$ and
\eqref{eq:comp2} yields $g_{3} = 3 \cdot K>0$.
Hence, to obtain a solution for which the complementarity conditions hold, i.e.\ to satisfy \eqref{eq:comp}, 
$y_{3}=0$ must hold while $y_{1} \geq 0$ and, in particular,  
$y_{1}=3$ so that constraint \eqref{eq:stat1} is satisfied.
From \eqref{eq:stat2} we obtain $y_{2}=1$. Hence, $g_{2}=0$ must hold which implies $x_{7}=x_{2}=3$. 
Summarizing, $x= \left(7, 3, 0, 0, 0, 0, 3 \right), y = \left(3, 1, 0 \right)$ is the unique optimal solution for any $K>0$.

We recast problem~\eqref{eq:bil} using the $L$-penalty reformulation~\eqref{eq:penL2}
with $L=L_{i}, i=1,2,3$,
and $K=10$: 
\begin{subequations} \label{eq:penalty}
\begin{alignat}{2}
\min \; &x_3 + x_4 + x_5 + x_6 
+ \frac{L}{2}  \left(y_{1}+ y_{2}+y_{3} -x_{1} + 8x_{2} + x_{7} +10 \right) \label{eq:objfunexo1}\\
s.t. \; &x_3 - x_4 = x_1-7 \label{eq:abs1b}\\
&x_5 - x_6 = x_2-3   \label{eq:abs2b}\\
&y_1  = 3 \label{eq:stat1b} \\
&y_{1}+y_2 - 10 y_3 = 4 \label{eq:stat2b} \\
&g_{1}(x,y) = 10 - x_1 - x_2 \geq 0  \label{eq:comp1b} \\
&g_{2}(x,y) = x_7 - x_2  \geq 0  \label{eq:comp3b} \\
&g_{3}(x,y) = 10 x_2  \geq 0 \label{eq:comp2b} \\
&x,y \geq 0.
\end{alignat}
\end{subequations}
For $L=1$ we obtain $x=(7,0,0,0,0,3,0)$ and $y=(3,1,0)$ as an optimal solution, which yields $g(x,y)=(3,0,0)$. Since $y_{1} \neq 0$ and $g_{1} \neq 0$, this solution does not satisfy complementarity, a contradiction to Theorem~2 of \citet{Gabriel2013}.

{\it What goes wrong in the proof of Theorem~2?}
In the proof from the original paper the KKT conditions for \eqref{eq:penL} are formulated. 
It is shown that, 
if $\vp_{i}$ and $\vm_{i}$ are strictly positive, 
$L_{i}=-(\lambda_{4})_{i}$ where $\lambda_{4}$ denotes the Lagrange multiplier associated with the constraint $\vp_{i}+\vm_{i}-\frac{y_{i}+g_{i}}{2}=0$. 
The term that corresponds to $L_{i}$ in the objective function is cancelled out and it is claimed that an equivalent optimization problem for this recast KKT system exists that contains the constraint $\frac{y_{i}+g_{i}}{2}=0$
instead of $\vp_{i}+\vm_{i}-\frac{y_{i}+g_{i}}{2}=0$.
Equivalence, however, is only true as long as the KKT solution remains feasible at optimality,
which is not the case when $\vp_{i}>0, \vm_{i}>0$ and $\frac{y_{i}+g_{i}}{2}=0$. 
For Example~\ref{ex:1} and  $L=1$, 
$y_{1}+g_{1} = 6$ and, therefore, the KKT solution of the original problem is not feasible.    

\subsection{Counter-example to Theorem~3}
The next counter-example shows that the $L$-penalty approach might yield a solution not satisfying complementarity for any $L>0$
in contradiction to Theorem~3. 
\begin{example}[Non-complementarity  for every $L >0$] \label{ex:noL}
\begin{equation} \label{ex:2} 
\begin{array}{ll} 
\min & - y \\
\mbox{s.t.} & y \leq 4 \\
& g(x,y)=10-2 y \geq 0\\
& y \geq 0\\
& y  \cdot g =0
\end{array}   
\end{equation}
Since $g \geq 2$, by complementarity the unique optimal solution (with respect to $y$) for \eqref{ex:2} is $y=0$. Consider the objective function of the $L$-penalty method: 
\begin{equation} \label{exo2pen} 
\displaystyle 
f(x,y)+L\cdot \fpen(x,y) = -y + \frac{L}{2} \left( y+10-2y \right)
=-\left(1+ \frac{L}{2}\right)y+5 L. 
\end{equation}
For every $L>-2$, the optimal solution of the $L$-penalty approach is $y=4$, combined with $g\ge2$ thus, $ y \cdot g  \neq 0$.
It can be easily verified that the example satisfies the assumptions of Theorem~3 from \citet{Gabriel2013} which, applied to this example, state that $\max \{ 5 -\frac{y}{2} : y \leq 4, g(x,y)=10-2y \geq 0, y \geq 0 \} $ admits a finite optimal solution with a positive objective function value. 
\end{example}
The failure of Theorems~2 and~3 of \citet{Gabriel2013}  is explained from a bicriteria perspective in the next section. 

\section{Bicriteria interpretation of the $L$-penalty formulation} \label{sec:bicrit}
We can interpret the $L$-penalty formulation~\eqref{eq:penL2} 
\begin{align*} 
\min \ & f(x,y) + L \cdot \sum_{i=1}^{n_y}   \frac{y_i + g_i(x,y)}{2} 
\\
 \text{s.t.} \  & 
(x,y)\in \Gamma
\end{align*}
with 
feasible set $\Gamma = \{(x,y) : (x,y) \in \Omega, y \ge 0,  g(x,y) \ge 0 \}$
and 
a scalar parameter~$L$ 
as the weighted-sum formulation of a bicriteria optimization problem of the form
\begin{align} \label{bicrit} 
\min \ & f(x,y) \notag \\ 
\min \ & \fpen(x,y) \\
 \text{s.t.} \  & 
(x,y)\in \Gamma \notag,
\end{align}
where 
\begin{equation} \label{eq:penaltyterm}
\fpen(x,y)=\displaystyle \sum_{i=1}^{n_y} \frac{y_{i}+g_{i}(x,y)}{2}.
\end{equation}
The objective function of \eqref{eq:penL2} 
consists of two terms: the original objective $f(x,y)$
and the $L$-penalty term $\fpen$ 
which is weighted by~$L$.
When the two objective functions $f$ and $\fpen$ are conflicting, there does not exist a single solution that optimizes both objectives simultaneously. 
Instead we deal with a set of solutions that cannot be improved with respect to one criterion without being deteriorated with respect to the other criterion. 
Such solutions are called {\it efficient}, their outcomes {\it nondominated}. 
In the following, we briefly state common notions from multicriteria optimization in the context of the specific bicriteria problem with objectives $f$ and $\fpen$.  
\begin{definition}[{Efficiency/Nondominance}]
Consider the bicriteria optimi\-zation problem~\eqref{bicrit}.
A solution $(\xb, \yb) \in \Gamma$ is called {\it efficient}, and its outcome
$(f(\xb, \yb),\fpen(\xb, \yb))^{\top}$ is called {\it nondominated} 
 if there does not exist a solution $(x,y) \in \Gamma$ such that
$f(x,y) \leq f(\xb, \yb)$ \emph{and} 
$\fpen(x,y) \leq \fpen(\xb, \yb)$
with at least one strict inequality.  
\end{definition}

\begin{definition}[Lexicographic minimum]
Let $(\xb, \yb) \in \Gamma$ be an optimal solution of 
$\min_{(x,y) \in \Gamma} f(x,y)$.
We call an optimal solution of 
\begin{align*} \label{eq:lexmin}
\min \ & \fpen(x,y)  \notag
\\
 \text{s.t.} \  
& f(x,y)=f(\xb, \yb)
\\
& (x,y)\in \Gamma \notag
\end{align*}
{\it lexicographic minimum} with respect to $(f,\fpen)$.
\end{definition}
The lexicographic minimum with respect to $(\fpen,f)$ is determined analogously. Note that every lexicographically optimal solution is efficient.
For more details on these basic notions from multicriteria optimization we refer to the textbooks of \citet{chankong83} and \citet{ehrgott05}.

If the problem is linear we can easily determine the nondominated set in the following way (see, e.g., \citet{aneja79} for a principle description of the procedure). 
Therefore, consider again the $L$-penalty formulation \eqref{eq:penalty} from Example~\ref{ex:1} for $K=10$.  

We first compute the two {\it lexicographic minima} with respect to $(f,\fpen)$ and $(\fpen,f)$.

Minimizing only $f$ yields 
$x_3 = x_4 = x_5 = x_6 = 0$, 
$x_1 = 7$ from \eqref{eq:abs1b} and 
$x_2 = 3$ from \eqref{eq:abs2b}. 
Then, $x_{7} \geq 3$ from \eqref{eq:comp3b}. 
Moreover, \eqref{eq:comp1b} and \eqref{eq:comp2b} are satisfied.   
Besides, 
$y_{1}=3$ from \eqref{eq:stat1b} and 
$y_{2}-10y_{3}=1$ from \eqref{eq:stat2b}.
Hence, the set of optimal solutions reads 
$x= \left(7, 3, 0, 0, 0, 0, \alpha \right), y = \left(3, 1+10\beta, \beta \right)$ with $\alpha \geq 3$ and $\beta \geq 0$. 
Fixing now this solution and considering the objective $\fpen$, thus, 
minimizing $(11y_{3} + x_{7}+31)/2$ with $x_{7} \geq 3$ and $y_{3} \geq 0$,
results in the efficient solution $x= \left(7, 3, 0, 0, 0, 0, 3 \right), y = \left(3, 1, 0 \right)$ for which $g=(0,0,30)$ and 
with optimal objective function value $\fpen=17$. 
Hence, $z^1=(0,17)$ is a nondominated point where the first component corresponds to $f$ and the second to $\fpen$.    
Note that it is by coincidence that this solution satisfies complementarity. 

In order to determine the second lexicographic minimum we first consider $\fpen$.
Since $y_{1}+y_{2} \geq 4$ from \eqref{eq:stat2b} and $y_{3}, g_{1}, g_{2},g_{3} \geq 0$, $f^{pen}\geq 2$ 
and $f^{pen}= 2$ is obtained for $y_{1}=3, y_{2}=1, y_{3}=0$ and $g_{1} = g_{2} = g_{3} = 0$.
This in turn implies
$x_{1}=10$ and $x_{2}=x_{7}=0$.
The values of $x_{3},x_{4},x_{5},x_{6}$ are not uniquely determined since these variables do not appear in the $L$-penalty objective. 
However, when (re-)optimizing 
the original objective over all solutions with $\fpen=2$, we obtain  
$x_{3}=x_{6}=3$, $x_{4}=x_{5}=0$.
Thus, $z^2=(6,2)$ with $f=6$ and $\fpen=2$ is a nondominated point 
with associated efficient solution $x=(10,0,3,0,0,3,0)$, $y=(3,1,0)$ and $g=(0,0,0)$.

The two lexicographic minima $z^1=(0,17)$ and $z^2=(6,2)$ give us a range of the outcomes. 
For any $L>0$ the value of $f$ will be in $[0,6]$ and the value of $\fpen$ in $[2,17]$. 
However, we do not need to test certain values of $L$ randomly. 
According to the procedure proposed in \citet{aneja79} we can use two adjacent nondominated points 
to derive a value for $L$ which enables us to search for a new nondominated point between the given two points in a systematic way. 
Based on $z^1=(0,17)$ and $z^2=(6,2)$ we compute
\[
L=\frac{z^2_{1}-z^1_{1}}{z^1_{2}-z^2_{2}}=\frac{6-0}{17-2}=0.4
\] 
From the solution of problem~\eqref{eq:penalty} with $L=0.4$ we numerically obtain the new nondominated point $z^3=(3,3.5)$. 
Again, we use pairs of adjacent nondominated points to update $L$.  
However, solving~\eqref{eq:penalty} with 
$L=(z^3_{1}-z^1_{1})/(z^1_{2}-z^3_{2})=2/9$ and
$L=(z^2_{1}-z^3_{1})/(z^3_{2}-z^2_{2})=2$, respectively, does not yield a new nondominated point.
Hence, the nondominated set consists of the three points $z^1=(0,17)$, $z^3=(3,3.5)$, $z^2=(6,2)$ and all convex combinations between the pairs of adjacent points $(z^1,z^3)$ and $(z^3,z^2)$.
A visualization 
is given in Figure~\ref{fig:exo1}. 

An overview of all possible values of $L$, the resulting objective function values and the solutions~$x$,$y$ as well as~$g$ is given in Table~\ref{tab:all}.   
\begin{table}
\centering
\[
 \begin{array}{c|cc|ccc}
  \toprule
L & f & f^{pen} & x & y & g\\
 \hline
 [0, \frac{2}{9}]  & 0 & 17 & 
 (7,3,0,0,0,0,3) &
 (3,1,0)&
 (0,0,30)
 \\[3pt]
 [\frac{2}{9},2] 
 & 3 & \frac{7}{2} & 
  (7,0,0,0,0,3,0) &
 (3,1,0)&
 (3,0,0)
 \\[3pt]
 [2,\infty) 
 & 6 & 2 & 
  (10,0,3,0,0,3,0) &
 (3,1,0)&
 (0,0,0)
\\
\bottomrule
\end{array}
\]
\caption{Solutions of the $L$-penalty problem~\eqref{eq:penalty} for different choices of $L$ with $K=10$. At the breakpoints $L=\frac{2}{9}$ and $L=2$ (infinitely many) alternative solutions exist.
}
\label{tab:all}
\end{table}
With increasing values of $L$ the minimization of~$\fpen$
becomes more important. 
Note that there exist infinitely many optimal solutions 
for $L=2/9$ and $L=2$. 
For example, for $L=2/9$, besides the two solutions indicated in Table~\ref{tab:all} all solutions $(x,y)$ with $x=(7,\lambda,0,0,0,3-\lambda,\lambda)$, $0 < \lambda <3,$
and $y=(3,1,0)$ are feasible and result in an overall objective function value of $34/9$.

\begin{figure}
\floatbox[{\capbeside\thisfloatsetup{capbesideposition={right,top},capbesidewidth=4cm}}]{figure}[\FBwidth]
{\caption{Nondominated set of Example~\ref{ex:1}.}\label{fig:exo1}}
{
\def\JPicScale{0.5} 
\psset{unit=\JPicScale mm}
\psset{linewidth=0.5,dotsep=1,hatchwidth=0.3,hatchsep=1.5,shadowsize=1,dimen=middle}
\psset{dotsize=2 2.5,fillcolor=black}
\psset{arrowsize=3,arrowlength=1,arrowinset=0.25,tbarsize=0.7 5,bracketlength=0.15,rbracketlength=0.15}

\psscalebox{.5}{%
\begin{pspicture}(10,10)(130,200)
\psline{->}(10,10)(130,10)
\psline{->}(10,10)(10,200)
\rput(125,0){\Large $f$}
\rput(-2,190){\Large$f^{pen}$}

\psline{*-}(10,180)(40,45)
\psline{*-*}(40,45)(70,30)
\rput(30,180){\Large$(0,17)^{\top}$}
\rput(60,50){\Large$(3,3.5)^{\top}$}
\rput(90,30){\Large$(6,2)^{\top}$}

\psline{-}(7,60)(13,60)
\psline{-}(7,110)(13,110)
\psline{-}(7,160)(13,160)
\rput(-2,60){\Large $5$}
\rput(-2,110){\Large$10$}
\rput(-2,160){\Large$15$}

\psline{-}(60,7)(60,13)
\psline{-}(110,7)(110,13)
\rput(60,-2){\Large$5$}
\rput(110,-2){\Large$10$}

\end{pspicture}
}
}
\end{figure}
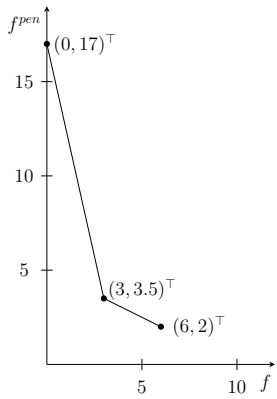

Further note that the optimal solution obtained for $L \in [0,\frac{2}{9})$ equals - by chance - the solution of the original complementarity problem~\eqref{eq:bil}. 
Thus, we state that for this particular problem the $L$-penalty approach and a variation of $L$ would not be required theoretically but complementarity constraint~\eqref{eq:comp} could simply be ignored. 
Nevertheless, this example serves to demonstrate the failure of Theorem 2 of \citet{Gabriel2013} which claims solutions 
satisfying complementarity for every $L>0$. In this example, however, for all $L \in  (\frac{2}{9},2)$ the optimal solution does not respect complementarity since $y_{1}>0$ and $g_{1}>0$.  

The two conflicting objectives are one reason for Theorem~2 to fail. 
We can not expect that complementarity is satisfied for every $L>0$.
A further issue is the fact that the $L$-penalty term as formulated in~\eqref{eq:penL} does not guarantee complementarity, in general.
For sufficiently large values of~$L$, solutions with $y$ and $g$ small 
are obtained, however, there is no guarantee to achieve complementarity as demonstrated by Example~\ref{ex:noL}.

\section{Obtaining solutions satisfying complementarity with Siddiqui and Gabriel's $L$-penalty approach} \label{sec:repair}

In Example~\ref{ex:noL} a solution satisfying complementarity exists, 
however, it is not optimal for \eqref{eq:penL2} for any $L>0$.  
Hence, the question arises under which conditions we can generate solutions satisfying complementarity with Siddiqui and Gabriel's $L$-penalty approach.
We have already stated that $\fpen$, which is non-negative by definition, is made smaller (all things being equal) with increasing values of $L$. 
In case a solution that corresponds to the minimum possible value of $\fpen$ satisfies complementarity we can expect it to be generated for a sufficiently large value of $L$.
Consider once more Example~\ref{ex:1} in which $\fpen \geq 2$ and the solution that realizes $\fpen=2$ satisfies complementarity. Indeed, this solution is generated for all $L>2$.
In Example~\ref{ex:noL}, $\fpen \geq 3$ but the solution that attains $\fpen=3$ is $y=4$, $g=2$, thus, does not satisfy complementarity. This is the reason why no solution for which the complementarity conditions hold is obtained with the $L$-penalty approach of \citet{Gabriel2013} in this example.     

In the following we show that a
lexicographic minimal solution with respect to $(\fpen,f)$ is generated for sufficiently large values of $L$
under certain conditions. 
In order to guarantee that parameter~$L$ is finite we need the notion of {\it proper efficiency} or {\it bounded trade-off} which is defined below for our specific bicriteria application. 
If the lexicographic minimal solution satisfies complementarity, we can then assure that the solution of the $L$-penalty approach satisfies complementarity for sufficiently large values of~$L$. 

\begin{definition}[Proper Efficiency/Nondominance]
A solution $(\xb,\yb) \in \Gamma$ 
of the bicriteria optimization problem~\eqref{bicrit}
is called {\it properly efficient} according to \citep{geoffrion68}
if it is efficient and 
if there exists a scalar $\bar{L}>0$ 
so that 
for all 
$(x,y) \in \Gamma$ satisfying 
$f(x,y)<f(\xb,\yb)$ 
and $\fpen(x,y)>\fpen(\xb,\yb)$   
\begin{equation} \label{tradeoff}
\frac{ f(\xb,\yb)-f(x,y)}{\fpen(x,y)-\fpen(\xb,\yb)} \leq \bar{L}
\end{equation}
and 
for all 
$(x,y) \in \Gamma$ satisfying 
$\fpen(x,y)<\fpen(\xb,\yb)$ 
and $f(x,y)>f(\xb,\yb)$   
\begin{equation} \label{tradeoff2}
\frac{ \fpen(\xb,\yb)-\fpen(x,y)}{f(x,y)-f(\xb,\yb)} \leq \bar{L}.
\end{equation}
\end{definition}
The quotients in \eqref{tradeoff} and \eqref{tradeoff2} are typically denoted as {\it trade-offs} between the two objectives \citep{chankong83}.

Figure~\ref{fig:th4} illustrates the notion of proper efficiency. 
In the subfigure on the right, the trade-off is not bounded in $z=(f(\xb,\yb),\fpen(\xb,\yb))$, while it is in the subfigure on the left which depicts a linear problem.  
Note that every efficient solution is properly efficient in the linear case.  
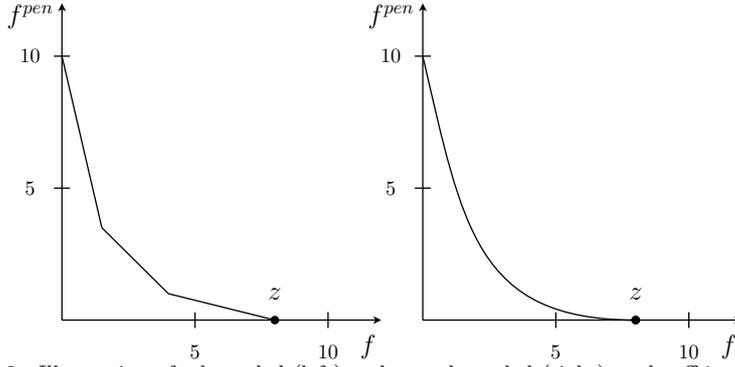
\begin{figure} 
\def\JPicScale{0.5} 
\psset{unit=\JPicScale mm}
\psset{linewidth=0.5,dotsep=1,hatchwidth=0.3,hatchsep=1.5,shadowsize=1,dimen=middle}
\psset{dotsize=2 2.5,fillcolor=black}
\psset{arrowsize=3,arrowlength=1,arrowinset=0.25,tbarsize=0.7 5,bracketlength=0.15,rbracketlength=0.15}

\begin{minipage}{.48\textwidth}
\hfill
\psscalebox{.7}{%
\begin{pspicture}(10,10)(130,130)
\psline{->}(10,10)(130,10)
\psline{->}(10,10)(10,130)
\rput(125,0){\Large $f$}
\rput(-2,125){\Large$f^{pen}$}

\psline{-}(10,110)(25,45)
\psline{-}(25,45)(50,20)
\psline{-*}(50,20)(90,10)
\rput(90,20){\Large$z$}

\psline{-}(7,60)(13,60)
\psline{-}(7,110)(13,110)
\rput(-2,60){\large $5$}
\rput(-2,110){\large$10$}

\psline{-}(60,7)(60,13)
\psline{-}(110,7)(110,13)
\rput(60,-2){\large$5$}
\rput(110,-2){\large$10$}
\end{pspicture}
}
\end{minipage}
\hfill
\begin{minipage}{.48\textwidth}
\psscalebox{.7}{%

\begin{pspicture}(10,10)(130,130)
\psline{->}(10,10)(130,10)
\psline{->}(10,10)(10,130)
\rput(125,0){\Large $f$}
\rput(-2,125){\Large$f^{pen}$}
\psline{-}(10,110)(10,110)
\psline{*-*}(90,10)(90,10)
\psbezier[showpoints=false]{-}%
(10,110)(25,45)(30,10)(90,10)
\rput(90,20){\Large$z$}

\psline{-}(7,60)(13,60)
\psline{-}(7,110)(13,110)
\rput(-2,60){\large $5$}
\rput(-2,110){\large$10$}

\psline{-}(60,7)(60,13)
\psline{-}(110,7)(110,13)
\rput(60,-2){\large$5$}
\rput(110,-2){\large$10$}
\end{pspicture}
}
\end{minipage}
\caption{
Illustration of a bounded (left) 
and 
an unbounded (right) trade-off in the nondominated point $z=(f(\xb,\yb),\fpen(\xb,\yb))$ 
} \label{fig:th4}
\end{figure}

\setcounter{theorem}{3}
\begin{theorem} \label{theo:new}
Assume that $\Omega$ and $\Gamma = \{(x,y) : (x,y) \in \Omega, y \ge 0,  g(x,y) \ge 0 \}$ 
are non-empty and compact. 
Let $(\xn,\yn) \in \Gamma$ be lexicographically minimal with respect to $(\fpen, f)$. 
Moreover, let $(\xn,\yn)$ be properly efficient with the trade-off bounded by a scalar $\bar{L}>0$.

Then $(\xn,\yn)$ is an optimal solution of~\eqref{eq:penL2} for every $L > \bar{L}$. 
Moreover, if $(\xn,\yn)$ satisfies complementarity
the optimal solution of~\eqref{eq:penL2} satisfies complementarity for every $L > \bar{L}$.
\end{theorem}

\begin{proof}
Since $(\xn,\yn) \in \Gamma$ is lexicographically minimal with respect to $(\fpen, f)$, it is efficient.
Moreover, by definition there is no  
$(x,y) \in \Gamma$ with $\fpen(x,y)<\fpen(\xb,\yb)$, hence, proper efficiency of this solution  
means that 
for all $(x,y) \in \Gamma$ satisfying 
$f(x,y)<f(\xn,\yn)$ and $\fpen(x,y)>\fpen(\xn,\yn)$ 
\[
\frac{ f(\xn,\yn)-f(x,y)}{\fpen(x,y)-\fpen(\xn,\yn)} \leq \bar{L}.
\]
We want to show that $(\xn,\yn)$ is an optimal solution of~\eqref{eq:penL2} for every $L > \bar{L}$.
Therefore, we consider the slightly modified objective function  
$$
\modf(x,y) = f(x,y) + L \cdot [\fpen(x,y)-\fpen(\xn,\yn) ] 
$$
which only differs from~\eqref{eq:penL2} by the constant term
$L \cdot \fpen(\xn,\yn)$, hence, yields the same optimal solution set.
Now, for every $(x,y) \in \Gamma$ 
with $f(x,y)<f(\xn,\yn)$ and $\fpen(x,y)>\fpen(\xn,\yn)$ 
and for every $L > \bar{L}>0$
\begin{align*}
\tilde{f}(x,y) &= f(x,y) + L \cdot [\fpen(x,y)-\fpen(\xn,\yn) ] \\
& >  f(x,y) + \bar{L} \cdot [\fpen(x,y)-\fpen(\xn,\yn) ] \\
&\geq f(x,y) + \frac{ f(\xn,\yn)-f(x,y)}{\fpen(x,y)-\fpen(\xn,\yn)} \cdot [\fpen(x,y)-\fpen(\xn,\yn) ] \\
& = f(\xn,\yn) \\ 
&= f(\xn,\yn) + L \cdot 0 \\
&= f(\xn,\yn) + L \cdot [\fpen(\xn,\yn)-\fpen(\xn,\yn) ]\\
&= \tilde{f}(\xn,\yn). 
\end{align*}
Since $(\xn,\yn)$ is lexicographically minimal with respect to $(\fpen, f)$, there is no 
$(x,y) \in \Gamma \backslash \{(\xn,\yn)\}$ with $\fpen(x,y)<\fpen(\xn,\yn)$
or $\fpen(x,y)=\fpen(\xn,\yn)$ and $f(x,y)<f(\xn,\yn)$. 
Hence, $(\xn,\yn)$ is minimal for~\eqref{eq:penL2} for every $L > \bar{L}>0$.   
Clearly, if $(\xn,\yn)$ satisfies complementarity,
an optimal solution of~\eqref{eq:penL2} satisfies complementarity for every $L > \bar{L}$.
\qed
\end{proof}

\begin{remark}
An optimal solution of \eqref{eq:penL2} obtained under the assumptions of Theorem~\ref{theo:new}  is not necessarily optimal for the original complementarity problem \eqref{eq:mpecorig}. As can be seen from Example~\ref{ex:1} there might be solutions satisfying complementarity with a smaller value of $f$.     
According to Table~\ref{tab:all} the lexicographic minimum with respect to $(\fpen,f)$ is
$\xn=(10,0,3,0,0,3,0)$, $\yn=(3,1,0)$ with corresponding $g=(0,0,0)$ and (original) objective function value $f(\xn,\yn)=6$. 
However, the solution  $x=(7,3,0,0,0,0,3)$, $y=(3,1,0)$ with corresponding $g=(0,0,30)$ also satisfies complementarity and yields $f(x,y)=0$.  
\end{remark}

The next example illustrates Theorem~\ref{theo:new}. 
\begin{example}[Complementarity for $L>2$] \label{ex:3}
Consider 
\begin{equation} \label{eq:counterexo2}
\begin{array}{ll} 
\min & -x_{3} - y_{2} \\
\mbox{s.t.} & x_{3} \leq  20\\
 & y_{2} \leq  10\\
& g_{1}(x,y)=10-x_{1}-x_{2} \geq 0\\
& g_{2}(x,y)=x_{3}-x_{2} \geq 0\\
& y_{1}, y_{2} \geq 0\\
& g^{\top} y =0\\
& x_{1}, x_{2}, x_{3} \geq 0
\end{array}   
\end{equation}
and the associated objective of the $L$-penalty formulation 
\begin{eqnarray} \label{eq:Lpenobj3}
f + L \cdot \fpen
&&= -x_{3} - y_{2} + \frac{L}{2} (y_{1} + y_{2} + 10-x_{1}-x_{2}  + x_{3}-x_{2})  \\ 
&& =  \frac{L}{2} y_{1} + \left( \frac{L}{2}-1 \right)y_{2} -\frac{L}{2}x_{1}-Lx_{2}  + \left( \frac{L}{2}-1 \right)x_{3} + 5L \notag
\end{eqnarray}
Note that the vectors $x=(0,10, 10)$, $y=(0,0)$  with corresponding $g_{1}(x,y)=0=g_{2}(x,y)$ are feasible for 
\eqref{eq:counterexo2}. 
It can be easily verified that this solution, which satisfies complementarity, represents the unique lexicographic minimal solution with respect to $(\fpen,f)$.
Since the problem is linear, every efficient solution and in particular 
the lexicographic minimum is properly efficient.
Hence, the assumptions of Theorem~\ref{theo:new} hold and there must exist a finite~$\bar{L}>0$ such that 
for every $L > \bar{L}$
an optimal solution of~\eqref{eq:penL2} equals the lexicographic minimum.
Moreover, since this solution satisfies complementarity, an optimal solution of~\eqref{eq:penL2} satisfies complementarity.

Indeed, for $L > 2$, $y_{1}, y_{2}$ and $x_{3}$ are chosen as small as possible and $x_{1}, x_{2}$ as large as possible (within their bounds), where the higher weight is given to $x_{2}$. 
Hence, the lexicographic minimal solution $x=(0,10, 10)$, $y=(0,0)$ with $g=(0,0)$ is optimal for all $L>\bar{L}$ with
$\bar{L}=2$.  

For $L\in[0,2)$ there is an incentive in~\eqref{eq:Lpenobj3} to choose $y_{2}$ and $x_{3}$ as large as possible. Thus, an optimal solution for $L\in[0,2)$ is
$x=(0,10,20), y=(0,10), g=(0,10)$. 
This solution does not satisfy complementarity. 
For $L=2$, a solution for which the complementarity conditions hold might be obtained but is not enforced by objective~\eqref{eq:Lpenobj3}. 
\end{example}

Finally, note that the assumption that a lexicographic minimal solution with respect to $(\fpen,f)$ satisfies complementarity
is sufficient but not necessary for the $L$-penalty method to work, in general.
To illustrate this consider the non-linear example presented below motivated from Section~3 of \citet{Gabriel2013}, 
which shows that the $L$-penalty method might work successfully also for selected values $L>0$ despite 
all resulting solutions with $\fpen>0$ are not lexicographically minimal with respect to $(\fpen,f)$.
\begin{example}[Complementarity  for every $L \geq 0$] \label{ex:origpaper}
Consider
\begin{alignat}{2}
\min \; & (y_{1}+y_{2}+x-12) \cdot x  \notag \\
s.t. \; & g_{1}(x,y) = x+2y_{1}+y_{2}-12  \geq 0 \notag \\
&g_{2}(x,y) = x+y_{1}+2y_{2}-12  \geq 0 \label{prob:origpaper}\\
&y^{\top} g = 0 \notag \\
&x,y_{1},y_{2} \geq 0. \notag
\end{alignat}
The corresponding $L$-penalty term reads
$$
f^{pen} = 2(y_{1}+y_{2})+x-12.
$$
Table~\ref{tab:allexo2} summarizes the (numerical) results for different values of~$L$. 
Again, we see that the minimization of the second objective~$f^{pen}$ gains importance for increasing values of~$L$.
Since $g=0$ is optimal for all evaluated $L \geq 0$,  complementarity holds for all these $L$, however, $\fpen=0$ is not satisfied at optimality for those sampled values of $L\leq 10$.
\end{example}

\begin{table}

\centering
\[
 \begin{array}{l|cc|cccc}
  \toprule
L & f & f^{pen} & x & y  & g\\
 \hline
 0  & -12 & 2 & 6 & (2,2) & (0,0) \\ 
 0.0001   & -12 & 2 & 6 & (2,2) & (0,0) \\ 
  1  & -11.92 & 1.83 & 6.5 & (1.83,1.83) & (0,0) \\ 
 10  & -3.67 & 0.33 & 11 & (0.33,0.33) & (0,0) \\ 
 100  & 0 & 0 & 12 & (0,0) & (0,0) \\ 
\bottomrule
\end{array}
\]
\caption{Numerical solutions of the $L$-penalty formulation of problem~\eqref{prob:origpaper} for different choices of $L$. 
All these solutions satisfy complementarity. 
}
\label{tab:allexo2}
\end{table}

\end{example}

As a practical consequence we propose to first remove the complementarity conditions from 
the original problem~$(3^{\star})$ 
and solve the resulting problem.
In case one obtains a solution satisfying complementarity as in Examples~\ref{ex:1} and~\ref{ex:origpaper} there is no need to introduce an $L$-penalty formulation or any other approach to reformulate the complementarity condition.

\section{Conclusions} \label{sec:conclusions}

This corrigendum shows with the help of appropriate counter-examples that Theorems~2 and~3 in \citet{Gabriel2013} do not hold, in general.
A bicriteria analysis helps to understand how the proposed $L$-penalty method works.
In particular, the parameter~$L$ can be interpreted as the weight of the penalty term with respect to the original objective function.
We present a new theorem that guarantees that a solution of the $L$-penalty formulation 
satisfies complementarity for sufficiently large values of $L$ under the condition that the minimal value of the penalty term is attained by a solution 
satisfying complementarity
and that the trade-off at this solution is bounded.
As also shown there are instances in which solutions for which the complementarity conditions hold exist but are not accessible by the $L$-penalty method of \citet{Gabriel2013} for any $L \geq 0$. 
Approaches for generating them are left as a subject for future research.

\bibliographystyle{spbasic}      

\begin{thebibliography}{5}
\providecommand{\natexlab}[1]{#1}
\providecommand{\url}[1]{{#1}}
\providecommand{\urlprefix}{URL }
\expandafter\ifx\csname urlstyle\endcsname\relax
  \providecommand{\doi}[1]{DOI~\discretionary{}{}{}#1}\else
  \providecommand{\doi}{DOI~\discretionary{}{}{}\begingroup
  \urlstyle{rm}\Url}\fi
\providecommand{\eprint}[2][]{\url{#2}}

\bibitem[{Aneja and Nair(1979)}]{aneja79}
Aneja~Y, Nair~K (1979) 
Bicriteria {T}ransportation {P}roblem. 
  Management Science 25:73--78

\bibitem[{Chankong and Haimes(1983)}]{chankong83}
Chankong~V, Haimes~Y (1983) 
Multiobjective Decision Making: Theory and Methodology. 
Elsevier Science Publishing, New York

\bibitem[{Ehrgott(2005)}]{ehrgott05}
Ehrgott~M (2005) 
Multicriteria Optimization. 
Springer, Berlin

\bibitem[{Geoffrion(1968)}]{geoffrion68}
Geoffrion~A (1968) Proper {E}fficiency and the {T}heory of {V}ector {M}aximization. 
Journal of Mathematical Analysis and its Applications
22:618--630

\bibitem[{Luo et~al.(1996)Luo, Pang, and Ralph}]{luo1996mathematical}
Luo Z, Pang J, Ralph D (1996) Mathematical Programs with Equilibrium
  Constraints. Cambridge University Press

\bibitem[{Siddiqui and Gabriel(2013)}]{Gabriel2013}
Siddiqui S, Gabriel S (2013) An {SOS1}-Based Approach for Solving MPECs with a
  Natural Gas Market Application. Networks and Spatial Economics 13(2):205--227

\end{thebibliography}

\end{document}